%% file: intermingle-7.27.2024.tex
\documentclass[11pt,reqno]{amsart}
\usepackage{amscd,amssymb,amsmath,amsthm}
\usepackage{mathabx}
\usepackage{graphicx}

\usepackage{graphicx}
\usepackage[english]{babel}
\usepackage[T1]{fontenc}
\usepackage[arrow,matrix]{xy}
\usepackage{graphicx,subfigure,tikz,tikz-3dplot}
\usetikzlibrary{positioning,matrix,arrows,calc}
\usepackage{cite}
\usepackage{dsfont}
\usepackage{enumitem}
\usepackage{bbm}%fuer die dicke Indikatorfunktionseins -> \mathbbm
\usepackage{bm}
\usepackage{hyperref}
\hypersetup{
	colorlinks=true,
	linkcolor=blue}
\usepackage{array}

\oddsidemargin=0.2in \evensidemargin=0.2in \theoremstyle{plain}
\newtheorem{theorem}{Theorem}[section]
\newtheorem{lemma}[theorem]{Lemma}
\newtheorem{proposition}[theorem]{Proposition}
\newtheorem{corollary}[theorem]{Corollary}

\theoremstyle{definition}

\newtheorem{example}[theorem]{Example}

\newcommand*{\medcap}{\mathbin{\scalebox{1.5}{\ensuremath{\cap}}}}%
\newcommand*{\medcup}{\mathbin{\scalebox{1.5}{\ensuremath{\cup}}}}%
\numberwithin{equation}{section} 

\newcounter{hypocounter}
\renewcommand\thehypocounter{(H\arabic{hypocounter})} % this produces '(H1)'
\begin{document}

\title[Thick attractors with intermingled basins]{Thick attractors with intermingled basins}

\author[Abbas Fakhari and Ale Jan Homburg]{Abbas Fakhari and Ale Jan Homburg}
 % \subjclass{37E05,  37H10,  37H15}

    \address{A. Fakhari\\ Department of Mathematics,
Shahid Beheshti University, 19839 Tehran, Iran}
  % \email{a_fakhari@sbu.ac.ir}
	\email{a\_fakhari@sbu.ac.ir}

	\address{A.J. Homburg\\ KdV Institute for Mathematics, University of Amsterdam, Science park 107, 1098 XG Amsterdam, Netherlands\newline Mathematical Institute, University of Leiden, PO Box 9512, 2300 RA Leiden, Netherlands}
	\email{a.j.homburg@uva.nl}


\begin{abstract}
We construct various novel and elementary examples of dynamics with {metric} attractors that have intermingled basins.
A main ingredient is the introduction of random walks along orbits of a given dynamical system.
We develop theory for it and use it in particular to provide examples of thick {metric} attractors with intermingled basins.
\end{abstract}
	
\maketitle

\section{Introduction}

Consider a continuous map $F:  X\to X$ on a metric space $X$, equipped with a probability (reference) measure  $\mu$.
A closed subset $A \subset  X$ will be called a {\em {metric} attractor} of $F$ if it satisfies two
conditions:
\begin{enumerate}

\item The {\em basin of attraction}
\[
 \rho(A) := \left\{  x \in X \; ; \; \omega(x) \subset A \right\}
\]
has positive measure $\mu (\rho (A)) >0$.

\item
There is no strictly smaller closed set $A' \subset A$ so that $\rho (A')$ coincides with $\rho(A)$
up to a set of measure zero.

\end{enumerate}
The basin of attraction is not required to be an open set. This notion of attractor was coined by Milnor \cite{MR0790735} to have a definition that is less restrictive
than topological definitions making use of asymptotic stability. One also finds the term Milnor attractor in the literature.
A {metric} attractor is called a {\em minimal ({metric}) attractor}
if moreover
\begin{enumerate}

\item[(3)]
 There is no strictly smaller closed set $A'\subset A$ for which $\rho(A')$ has positive
measure.

\end{enumerate}

We also adopt from  \cite{MR0790735} the definition of
{\em likely limit set} $A$ of $F$ as the  smallest closed subset of $X$ with the
property that $\omega (x) \subset A$ for every point $x \in X$ outside of a set of measure zero.

The basins of attraction of two attractors $A_1$ and $A_2$ are said to be {\em intermingled basins of attraction}  if they are measure theoretically dense in each other: if one basin meets an open set $U$ in a set of positive measure then also the other basin meets $U$ in a set of positive measure \cite{MR1206103}.
A key example of a system with intermingled {metric} attractors is due to Kan \cite{MR1254075}.
It involves a smooth map on the annulus $\mathbb{T} \times [0,1]$, that fixes the boundaries
$\mathbb{T} \times \{0\}$ and $\mathbb{T} \times \{1\}$. The boundaries in this example are {metric} attractors.
An explicit expression for such a map is
\[
(u,x) \mapsto \left(3 u \pmod 1 , x + \frac{\cos(2 \pi u)}{32} x (1-x) \right).
\]
% See \cite{MR2434457,MR3184818,MR2105774,MR3721879,MR3600648,MR3926097,MR4313711}
% for further studies of Kan's example.
Various other constructions have been provided in the literature, for instance
% \cite{MR2183303,MR2071239,MR3799757,MR3452277,MR1406437,MR1406426}.
\cite{MR2183303,MR2477416,MR2071239,MR3799757}.
Examples of the occurrence of attractors with intermingled basins in models
are in \cite{PhysRevE.85.036207,MR3999561,MR2102760}.

To address measure theoretical aspects of Kan's example, we need two more notions.
Recall first that for an ergodic invariant measure $\nu$ of $F$, the  basin of $\nu$ is defined by
$$
\rho(\nu) := \{x\in X \; ; \; \lim_{n\to \infty}\frac{1}{n}\sum_{i=0}^{n-1}\delta_{F^i(x)}\rightarrow\nu\},
$$
with convergence in the weak star topology and where $\delta_x$ stands for the atomic measure at $x$.
For the second notion, we need that $X$ is a smooth manifold and $F$ is a smooth map defined on it. By the classical Oseledec's theorem, if $\nu$ is an ergodic $F$-invariant measure then for $\nu$-a.e $x\in X$ and any vector $v$ belonging to the tangent space $T_x X$, the limits
$$\lim_{n\to\infty} \frac{1}{n} \log \|DF^n(x)v\|$$
exist. The measure $\nu$ is called {\it hyperbolic} if the above limits are all non-zero.

In Kan's example, the Lebesgue measures on the two boundary circles are ergodic invariant measures with positive horizontal and negative vertical Lyapunov exponents. The negative Lyapunov exponents guarantee that the  basins have positive volume. The supports of both measures have zero volume, as they are supported on the boundary circles. In particular, these two measures are singular with respect to the volume on $\mathbb{T}\times [0,1]$. Now, a question arises.

\medskip
\noindent{\bf Question.}
Is there a dynamical system with two attractors with hyperbolic ergodic absolutely continuous invariant measures whose basins are intermingled in some open set?

\medskip
There are obstacles for an example in the context of partially hyperbolic dynamics. For instance, in \cite{MR3742552} the authors prove that each ergodic invariant measure of a partially hyperbolic diffeomorphism on $\mathbb{T}^3$ with intermingled basins is supported on a two dimensional torus and so it can not be absolutely continuous. In this context, we also have the stable and unstable saturation of dynamics on the support of an absolutely continuous measure as a crucial obstacle (see \cite{MR2851905}). In contrast, using the Anosov-Katok method, the author in \cite{MR?} provides an example of a $C^\infty$ diffeomorphism with 
two intermingled basins but with one hyperbolic component (open and dense basin) and the other having zero exponent.

The above question is a motivation for us to consider a class of attractors, called {\em thick metric attractors},  that have the potential to be the support of absolutely continuous measures. A metric attractor $A$ is thick
if $0 < \mu (A) < 1$, so $A$ has positive but not full measure.
This phenomenon of a thick attractor
has been described by Ilyashenko in \cite{MR2644340} who gave constructions for skew product systems over shifts,
arising from iterated function systems,  and later in \cite{MR2853609} also for diffeomorphisms.

Here, we begin by presenting fundamental ideas for creating attractors with intermingled basins. One of the main yet simple idea is random walks along orbits of a given dynamical system.
Random walks along orbits appeared  in
 \cite{MR1675369,MR1782584,MR1766346,MR1784208} that looked at contexts of ergodic measure-preserving automorphisms.
One setting is that of a diffeomorphism $f$, where
random walks along orbits amounts to
\[
x_{n+1} := f^{\eta_n} (x_n)
\]
for steps $\eta_n = \pm 1$.
We will allow the distribution of the steps $\eta_n$ to be position dependent, that is, depending on $x_n$.
For a flow $x(t) = \varphi(t,x_0)$ of a differential equation we will look at
\[
x(t) = \varphi  (s_t , x_0)
\]
where $s_t$ is a random process on $\mathbb{R}$.

We rewrite such random systems as deterministic systems to give novel and elementary constructions of attractors with intermingled basins.
Then, we will use the large freedom the constructions allow for to provide examples of thick {metric} attractors with intermingled basins. Finally, we describe how this strategy can lead to the creation of multiple metric attractors with intermingled basins.

\section{Elementary examples of intermingled basins}

We present different but related constructions of dynamics with {metric} attractors and intermingled basins.
We start with definitions of spaces and maps, also in order to fix notation.
Random binary choices are modeled by shifts on sequences of symbols.
For this we denote
\begin{align}
\label{e:sigma2}
\Sigma_2^+ &:= \{0,1\}^\mathbb{N},
\end{align}
equipped with the product topology.
Elements of $\Sigma_2^+$ are written as $\omega = (\omega_i)_{i\in\mathbb{N}}$.
The left shift operator $\sigma: \Sigma_2^+ \to \Sigma_2^+$ is given by $(\sigma \omega)_i = \omega_{i+1}$.
We write $[a_0,\ldots,a_k]$ for the cylinder
\[
 [a_0,\ldots,a_k] := \{   \omega \in \Sigma_2^+ \; ; \; \omega_i = a_i, 0\le i \le k\}.
\]
The number $k+1$ of fixed symbols is the depth of the cylinder. For two cylinders $C :=  [a_0,\ldots,a_k]$ and $D :=  [b_0,\ldots,b_l]$ we write $C \, D :=  [a_0,\ldots,a_k,b_0,\ldots,b_l]$.
Bernoulli measure $\nu_p$ corresponding to a probability $p$ for the symbol $0$ and $1-p$ for the symbol $1$ is determined by defining it on cylinder sets as
\[
\nu_p \left([a_0,\ldots,a_k]\right) := p^{\# \{i \; ; \; 0\le i \le k, a_i = 0\}   }  (1-p)^{{\# \{i \; ; \; 0\le i \le k, a_i = 1\}   }}.
\]
% We simply write $\nu$ for $\nu_{1/2}$.
For two-sided sequences we write $\Sigma_2 := \{0,1\}^\mathbb{Z}$, taking away the superscript "+".

Recall that two maps $F: X\righttoleftarrow $ and $G: Y\righttoleftarrow$ that preserve probability measures $\mu$ and $\nu$
are measurably isomorphic if there is a measure preserving bijection $H: X' \to Y'$ between
invariant subsets $X' \subset X$ and $Y' \subset Y$ of full measure  so that
$H \circ F = G \circ H$, see
\cite[Definition~2.7]{MR2723325}.
From the correspondence between binary expansions of numbers in the unit interval and symbol sequences in $\Sigma_2$, one gets that
the shift on $\Sigma_2^+$ with Bernoulli measure $\nu_p$ is
measurably isomorphic to the
piecewise linear map
$E_p: [0,1] \to [0,1]$ with Lebesgue measure, where $E_p$ is given by
\begin{align*}
% \label{e:Ex}
E_p (u) &:= \left\{ \begin{array}{ll} \displaystyle \frac{u}{p}, & 0 \le u < p, \vspace{0.1cm}
                                   \\  \displaystyle \frac{u-p}{1-p}, & p \le u \le 1.
                   \end{array}    \right.
\end{align*}
Note that indeed $E_p$ leaves Lebesgue measure on $[0,1]$ invariant.
Likewise, the shift on $\Sigma_2$ endowed with Bernoulli measure $\nu_{p}$ is measurably isomorphic to the two-dimensional baker map $B_p$ on $[0,1)^2$ given by
\begin{align}\label{e:baker}
B_p(w,y) &:=   \begin{cases}
	\displaystyle \left(\frac{w}{p} ,  p y \right), & 0\le w < p, \\
	\\
	\displaystyle \left(\frac{w}{1-p} -\frac{p}{1-p}  , (1-p) y + p \right), &  p  \le w < 1,
\end{cases}
\end{align}
and endowed with Lebesgue measure.

We will use the single symbol $\lambda$ to denote Lebesgue measure on the interval $[0,1]$ or the circle $\mathbb{T}$. We also write $\lambda$ for Lebesgue measure on higher dimensional
intervals $[0,1]^d$ and tori $\mathbb{T}^d$.

\subsection{Random walks on orbits of diffeomorphisms: position dependent probabilities}
\label{s:pdp}

Consider a compact manifold $N$ and a diffeomorphism $f: N \righttoleftarrow$.
Let $p:N \to (0,1)$ be a continuous function.
For a given initial point $x_0 \in N$, take a random walk
\begin{align}
\label{e:rw}
x_{n+1} &:= f^{\eta_n} (x_n)
\end{align}
along the orbit of $x_0$, where $\eta_n = \pm 1$ for $n \in \mathbb{Z}$ is taken independently with probabilities $p(x_n)$ for the value $+1$ and $1-p(x_n)$ for the value $-1$.
One can view this as random walks in environments on $\mathbb{Z}$, where an environment stands for a choice of probabilities on  points in $\mathbb{Z}$ with which possible steps are taken \cite{MR0217872}. Perturbations of initial points can then give different environments.
We get $x_n = x_n(\eta)$ as a function of $\eta = (\eta_i)_{i\in\mathbb{Z}}$.

We will write the above random walk as a skew product system over a shift.
Instead of taking two symbols $0$ and $1$ as in \eqref{e:sigma2} we find it convenient to use symbols $-1$ and $+1$ and write
\begin{align*}%\label{e:Omega}
\Omega^+ &= \{-1,+1\}^\mathbb{N}.
\end{align*}
Other notation such as for the shift operator and cylinders will be as before.
Define the skew product map
%$G : \Omega^+ \times N \to \Omega^+ \times N$
$G : \Omega^+ \times N \righttoleftarrow$
by
\begin{align}\label{e:G}
G(\eta , x) &:= \left(\sigma \eta , f^{\eta_0} (x) \right).
\end{align}
For iterates we use notation
\begin{align}\label{Def:S}
G^n (\eta,x) &=  \left(\sigma^n \eta , f^{S_n (\eta)} (x) \right), \text{ with } S_n(\eta) := \sum_{i=0}^{n-1} \eta_i.
\end{align}

Write $p_{-1} (x) := 1 - p (x)$ and $p_{1} (x) := p(x)$.
Let $\zeta_x$ be the measure on $\Omega^+$ which is defined on cylinders by
\begin{align}
\label{e:nux}
\zeta_x ([a_0\cdots a_k]) &:=  \prod_{i=0}^k p_{a_i} (f^{S_i (a)} (x)).
\end{align}
In words,  $\zeta_x([a_0\cdots a_k])$ equals the probability for the random walk starting at $x$ to walk through $x, f^{a_0} (x), f^{a_0+a_1} (x),\ldots,f^{a_0 + \cdots + a_k}(x)$.
By the Hahn-Kolmogorov extension theorem, \eqref{e:nux} defines the measure $\zeta_x$ on the Borel $\sigma$-algebra on $\Omega^+$. For a  measure $m$ on $N$, define the measure $\mu_m$ on $\Omega^+ \times N$ by
\begin{align}\label{e:mum}
\mu_m (A) &:= \iint_{\Omega^+ \times N}  {\mathbbm{1}}_{A_x}  (\omega) \,d\zeta_x (\omega) dm (x),
\end{align}
where $A_x := A \medcap (\Omega^+ \times \{x\})$.
As $p$ is continuous we find that $\zeta_x$ depends continuously on $x$ in the weak star topology.

A Borel measure $m$ on $N$ is stationary if
\begin{align}\label{e:mstat}
m(I)
%&= \int_{N} p_0 (f_0^{-1} (x))   \mathbbm{1}_I (f_0 (x)) +  p_1 (f_1^{-1} (x))  \mathbbm{1}_I %(f_1 (x)) \, dm (x),
%\\
&=  \int_{f^{-1} (I)} p_1 (x) \, dm (x)  + \int_{f (I)} p_{-1} (x) \, dm (x).
\end{align}
A theory of deterministic representations for
the random system \eqref{e:rw} is developed in \cite{MR2429852}, relating stationary measures to invariant measures of a certain deterministic system defined on $[0,1]\times N$. Our approach centers around the skew product system $G$ and has the following correspondence result
between stationary measures and invariant measures for $G$. Although formulated for the specific setup of iterated function systems generated by $f$ and $f^{-1}$, it is not restricted to this setup and works
more generally for iterated function systems with position dependent probabilities.
Although stationary measures do not play a vital role in the constructions in this paper, the following result is included as it clarifies the setup.

\begin{proposition}
A probability measure $m$ is stationary for the Markov process if and only if
 $\mu_m$ is invariant for $G$.
\end{proposition}

\begin{proof}
We first prove that if $\mu_m$ is invariant for $G$, then $m$ is stationary.
Consider a product set $A = C \times I$ of a cylinder $C$ and a Borel set $I$.
By definition of the measures $\zeta_x$ we have
\begin{align}
\nonumber
\zeta_{x} (-1\, C) &=  p_{-1} (x) \zeta_{f^{-1}(x)} ( C  ),
\\
\label{e:defnux}
 \zeta_{x} (1\, C) &= p_1 (x)  \zeta_{f(x)} ( C  ).
\end{align}
Calculate
\begin{align*}
\mu_m (G^{-1} (A))
&= \int_{f^{-1} (I)}  \zeta_{x} (1\, C) \, dm (x)  + \int_{f (I)} \zeta_{x} (-1\, C) \, dm (x)
\\
&=
\int_{f^{-1} (I)} p_1 (x)  \zeta_{f(x)} ( C  )  \, dm (x)  + \int_{f (I)} p_{-1} (x) \zeta_{f^{-1}(x)} ( C  ) \, dm (x)
\end{align*}
and note
\begin{align*}
\mu_m (A)
&= \int_I \zeta_x (C) \, dm (x).
\end{align*}
By invariance of $\mu_m$, these expressions are equal. Apply this to $C = \Omega^+$.
Then
\begin{align*}
m(I)
&= \int_{f^{-1} (I)} p_1 (x)    \, dm (x)  + \int_{f (I)} p_{-1} (x)  \, dm (x)
\end{align*}
which means that $m$ is stationary.

For the other direction, let $A := C \times I$ be a product set in $\Omega^+ \times N$ as above and suppose $m$ is a stationary probability measure.
Write \eqref{e:mstat} as
\begin{align*}
\int   \mathbbm{1}_I (x)\, dm (x)
&=  \int  \mathbbm{1}_I (f (x) ) p_1 (x) \, dm (x)  + \int \mathbbm{1}_I (f^{-1}(x)) p_{-1} (x) \, dm (x)
\end{align*}
Approximating an arbitrary integrable function $\psi$ by step functions we get from this
\begin{align}\label{e:psi}
\int   \psi (x)\, dm (x)
&=  \int  \psi (f (x) ) p_1 (x) \, dm (x)  + \int \psi (f^{-1}(x)) p_{-1} (x) \, dm (x).
\end{align}

Using \eqref{e:defnux} calculate
\begin{align*}
\mu_m (G^{-1} (A))
&=  \iint \mathbbm{1}_{1\, C} (\omega)  \mathbbm{1}_I (f (x) )  \, d\zeta_x (\omega) dm (x)  + \iint  \mathbbm{1}_{-1\, C} (\omega) \mathbbm{1}_I (f^{-1}(x))  \, d\zeta_x (\omega)  dm (x)
\\
&= \int  \zeta_x (1\, C) \mathbbm{1}_I (f (x) )  \,  dm (x)
+\int  \zeta_x (-1\, C)  \mathbbm{1}_I (f^{-1}(x))  \,   dm (x)
\\
&= \int  p_1(x) \zeta_{f(x)} (C)    \mathbbm{1}_I (f (x) )  \,  dm (x)
+\int  p_{-1} (x) \zeta_{f^{-1}(x)} (C)  \mathbbm{1}_I (f^{-1}(x))  \,   dm (x).
\end{align*}
Choosing  $\psi (x) := \zeta_x (C) \mathbbm{1}_I (x)$, we get from \eqref{e:psi}
that the last line equals
\[ \int \zeta_{x} (C)    \mathbbm{1}_I (x )  \,  dm (x), \]
which is $\mu_m(A)$.
So
\[
 \mu_m (G^{-1}(A))  = \mu_m (A),
\]
which proves invariance of $\mu_m$ since these product sets generate the Borel $\sigma$-algebra.
\end{proof}

\subsubsection{North pole/south pole diffeomorphisms on the circle}

We specialize the above construction to a north pole/south pole diffeomorphism
$f$ on a circle $\mathbb{S}^{1}$. This gives a prototype example of {metric} attractors with intermingled basins.
 The north pole $p_N$ is a repelling fixed point for $f$  and the south pole $p_S$ an attracting fixed point. The basin of attraction of $p_S$ is $\mathbb{S}^{1} \setminus \{p_N\}$.

Assume the function $p$ satisfies
\begin{align}
\label{e:p}
p (p_S) > 1/2, \qquad p(p_N) < 1/2.
\end{align}
As before, denote Lebesgue measure on the circle by $\lambda$. The measure $\mu_\lambda$ on $\Omega^+ \times \mathbb{S}^1$
is then defined  as in \eqref{e:mum}.

\begin{theorem}\label{t:northsouthintermingled}
Let the continuous function $p:\mathbb{S}^1 \to (0,1)$ satisfy \eqref{e:p} and
take $\Omega^+ \times \mathbb{S}^1$ endowed by the reference measure $\mu_\lambda$.
Consider the dynamical system on $\Omega^+ \times \mathbb{S}^1$ given by \eqref{e:G}.
Then $\Omega^+ \times \{p_N\}$ and $\Omega^+ \times \{p_S\}$ are {metric} attractors with intermingled basins.
The union $\Omega^+ \times \{p_N,p_S\}$ is the likely limit set.
\end{theorem}

\begin{proof}
Consider an orbit $x_{n+1} := f(x_n)$.
Identifying $x_n$ with $n$, the random walk \eqref{e:rw} gives a random walk on $\mathbb{Z}$.
The statements follow from the theory of random walks in environments, see  \cite[Section~I.12]{MR0217872} and \cite[Section~2.2]{MR4035729}.
For every $x \in \mathbb{S}^1 \setminus \{p_S,p_N\}$, there is a $P(x) \in (0,1)$ so that the probability of converging
to $p_S$ equals $P(x)$ and the probability of converging to $p_N$ is $1-P(x)$.
This means that
\begin{align*}
\zeta_x \left( \left\{  \eta \in \Omega^+ \; ; \; \lim_{n\to \infty} f^{S_n(\eta)} (x) = p_S   \right\} \right) &= P(x),
\\
\zeta_x \left( \left\{  \eta \in \Omega^+ \; ; \; \lim_{n\to \infty} f^{S_n(\eta)} (x) = p_N   \right\} \right) &= 1 - P(x),
\end{align*}
where, $S$ is defined as in (\ref{Def:S}). So  the basins of attraction $\rho \left(\Omega^+ \times \{p_S\}\right)$ and $\rho \left(\Omega^+ \times \{p_N\}\right)$ satisfy
\begin{align*}
\zeta_x \left( \rho \left(\Omega^+ \times \{p_S\}\right) \medcap \left(\Omega^+ \times \{x\}\right) \right) &= P (x),
\\
  \zeta_x \left( \rho \left(\Omega^+ \times \{p_N\}\right) \medcap \left(\Omega^+ \times \{x\}\right) \right) &= 1-P (x).
\end{align*}

Let $U$ be an open set in $\Omega^+ \times \mathbb{S}^1$. Take a product set $C \times I \subset U$ of a cylinder $C$ and an open interval $I \subset S^1$. If $C$ has depth $k$, then
the iterate $G^k ( C \times I)$ will be of the form $\Omega^+ \times I_k$ for an open interval $I_k$.
This gives
\begin{align*}
\mu_\lambda ( \rho \left(\Omega^+ \times \{p_S\}\right) \medcap \Omega^+ \times I_k  )  &= \int_{I_k} P(x)  \, d\lambda (x) >  0,
\\
 \mu_\lambda ( \rho \left(\Omega^+ \times \{p_N\}\right) \medcap \Omega^+ \times I_k) &= \int_{I_k} 1 - P(x) \, d\lambda(x) (x)> 0.
\end{align*}
So also both $ \mu_\lambda ( U \medcap \rho \left(\Omega^+ \times \{p_S\}\right))  > 0$
and
 $\mu_\lambda ( U \medcap \rho \left(\Omega^+ \times \{p_N\}\right)) > 0$.
This proves the statements.
\end{proof}

% The construction in this section seems to lend itself well for quantitative investigations in the  of \cite{MR3600648,MR3188635}.

\subsection{Skew product systems}
\label{s:circle}

We continue with elementary constructions of intermingled basins somewhat in the spirit of Kan's example \cite{MR1254075}, but using random walks on orbits of flows on compact manifolds.

Consider the flow of a Morse-Smale gradient differential equation
\begin{align}
\label{e:odef-p}
\dot{x} &= g(x)
\end{align}
on a compact manifold $N$.
We have
\begin{align}
\label{e:odegradh}
g &=  -\text{grad}\, (h)
\end{align}
with a height function $h : N \to [0,1]$. See for instance \cite{MR0669541} for background.
We assume that \eqref{e:odef-p} has a unique attracting equilibrium $p_S$ and a unique repelling equilibrium $p_N$. Such flows exist on any compact manifold.  We can assume $h(p_S) = 1$ and $h (p_N)=0$.
Note that flows that
provide north pole/south pole diffeomorphisms on the circle or on a sphere are possible examples.
Write $\varphi_t: N \to N$ for the flow of \eqref{e:odef-p}. We will also write $\varphi(t,x) := \varphi_t (x)$.

Let $E: \mathbb{T} \to \mathbb{T}$ (here $\mathbb{T} := \mathbb{R}/\mathbb{Z}$) be an expanding map $E (u) := L u \pmod 1$, for some integer
$L > 2$.
Let $s : \mathbb{T} \times N \to \mathbb{R}$ be a smooth scalar function satisfying  the following properties.

\begin{enumerate}

\item There are fixed points $q_1, q_2$ for $E$ so that for all $x \in N$,
\label{i:disc:0}
\[ s(q_1,x) <0, \qquad s(q_2,x)>0.
\]

\item We have
\[
\int_{\mathbb{T}}  s(u,p_N) \, du  < 0,
\qquad
\int_{\mathbb{T}}  s(u,p_S) \, du  > 0.
\]
 \label{i:disc:3}

\end{enumerate}
By compactness of $N$ there is $C>0$ so that for all $x \in N$,  $s(q_1,x) < -C < 0$ and  $s(q_2,x) > C > 0$.
An example of a system with the above conditions is obtained by taking a smooth scalar function $\eta: \mathbb{T} \to \mathbb{R}$ satisfying
\[
\int_{\mathbb{T}}\eta (u) \, du = 0, \qquad  \eta (q_1) < 0, \qquad \eta(q_2) >0,
\]
 and then letting
 \[
s (u,x) := \eta(u) + \delta  \left(h(x) - \frac{1}{2}\right)
\]
for some small $\delta>0$.

Consider the skew product system
%$F:  \mathbb{T} \times N \to \mathbb{T} \times N$
$F:  \mathbb{T} \times N \righttoleftarrow$
given by
\begin{align}\label{e:F}
F(u,x) &:= \left( E (u) ,  \varphi_{s(u,x)}  (x)      \right).
\end{align}
Write this as
\[
F(u,x) =  \left( E (u) ,  f_u  (x)      \right),
\]
and denote iterates as
\[
F^n(u,x) =  \left( E^n (u) ,  f^n_u  (x)      \right).
\]
Note $f^n_u(x) = \varphi_{s_n(u,x)} (x)$ with
\[
s_n(u,x) := \sum_{i=0}^{n-1}  s(E^i (u) , f^i_u(x) ).
\]

The top fiber Lyapunov exponent $L_{p_S} := \lim_{n\to \infty} \frac{1}{n} \ln \left( \left\| Df^n_u (p_S)\right\| \right)$ at $p_S$
exists for Lebesgue almost all $u\in\mathbb{T}$.
Applying Birkhoff's ergodic theorem, one gets
\[
\lim_{n\to \infty} \frac{1}{n} \sum_{i=0}^{n-1} s  (E^i(u),p_S) = \int_\mathbb{T}   s (u , p_S)\, du > 0.
\]
Using this in the linearized flow
$v\mapsto D \varphi_{s(u,p_S)} (p_S) v = e^{ Dg(p_S) s(u,p_S)  }v $ shows
\[
L_{p_S} =  \lim_{n\to \infty} \frac{1}{n} \ln \left( \left\| e^{Dg (p_S) n} \right\| \right)  \int_{\mathbb{T}}  s (u , p_S)\, du < 0.
\]
Here $\lim_{n\to \infty} \frac{1}{n} \ln \left( \left\| e^{Dg (p_S) n} \right\| \right) $ is the top Lyapunov exponent of the flow $v\mapsto D \varphi_{t} (p_S) v = e^{ Dg(p_S) t  }v$.
Likewise we have a negative top fiber Lyapunov exponent at $p_N$.

\begin{theorem}\label{t:ifsintermingledT}
Consider the skew product system $F$ from \eqref{e:F} on  $\mathbb{T} \times N$. Take volume  on $\mathbb{T} \times N$ as reference measure.
The sets $\mathbb{T} \times \{p_N\}$ and $\mathbb{T} \times \{ p_S\}$ are minimal {metric} attractors
with intermingled basins. The union $\mathbb{T} \times \{p_N,p_S\}$ is the likely limit set.
\end{theorem}

\begin{proof}
For $\delta$ positive, write $B_\delta (p_S)$ for the $\delta$-neighborhood of $p_S$.
For $\delta$ sufficiently small, say $\delta \le \delta_0$, $B_\delta (p_S)$ will be a ball.
Let $\lambda<0$ be such that the spectrum of $Dg (p_S)$ is contained in $\{ z \in \mathbb{C} \; ; \;  \text{Re}\, (z) < \lambda \}$.
Standard estimates on solutions of differential equations near equilibria show that for orbits $f^n_u(x)$ that stay inside $B_\delta (p_S)$, so
$f^i_u(x) \in  B_\delta (p_S)$ for $0\le i \le n$, we have a bound
\begin{align*}
d(f^n_u (x), p_S) &=
d \left( \varphi_{ s\left(E^n(u) , f^{n-1}_{E^n(u)} (x)\right) } \left(f^{n-1}_{E^n(u)} (x) \right)   \circ \cdots \circ \varphi_{s(u,x)}  (x) ,p_S\right)
\\
&\le C e^{\lambda \,s_n \left(u,p_S\right)} d\left(x,p_S\right)
\end{align*}
for some $C>0$.
The function
\[
r (u) =  \sup \left\{  \delta   \; ; \;  \delta \le \delta_0,  f^n_u (x_0) \in B_{\delta_0} (p_S) \text{ for } n \ge 0,
 \lim_{n\to\infty} f^n_u (x_0) = p_S \right\}
\]
 is therefore positive
for Lebesgue almost all $u$.
Compare the exposition of Pesin theory in \cite{MR4574839}, and the direct estimates as in  \cite[Lemma~3.1]{MR3600645},
\cite[Lemma~2.2]{MR1254075} or \cite[Lemma~A.1]{MR2477416}.
We thus find local stable manifolds $\{u\} \times  U_u \subset \{u\} \times N$ of $(u,p_S)$, for Lebesgue almost all $u$. More precisely, $x_0 \in U_u$
means that $F^n (u,x_0) \in \mathbb{T}\times B_{\delta_0} (p_S)$, for all $n \ge 0$ and hence the $\omega$-limit set
of $( u,x_0)$ is contained in  $\mathbb{T}\times \{p_S\}$.
Likewise there are local stable manifolds $\{u\} \times  V_u$ of $(u,p_N)$, for
which $(u,x_0)$, $x_0 \in V_u$, has its $\omega$-limit set in $\mathbb{T} \times \{p_N\}$.

Take an open set $U \subset \mathbb{T} \times N$. Because $E$ is an expanding map on $\mathbb{T}$, there is an iterate $F^j (U)$ so that the projection to the first coordinate in $\mathbb{T}$ of  $F^j (U)$ is surjective.
In particular $F^j (U)$ intersects the fiber $\{q_2\} \times N$ with $q_2$ as in property \eqref{i:disc:0}.
Using $s(q_2,x) > 0$, for all $x\in N$, shows that
there are iterates $F^n (U)$ that intersect the union
\[
 W^s (\mathbb{T}\times \{p_S\}) :=   \bigcup_{u \in \mathbb{T}} \{\omega\} \times U_u
\]
 of local stable manifolds of $(u,p_S)$ in a set of positive measure. Likewise there  are iterates $F^n (U)$ that intersect the union
\[
 W^s (\mathbb{T}\times \{p_N\}) :=
\bigcup_{u \in \mathbb{T}} \{\omega\} \times V_u
\]
of local stable manifolds of $(u,p_N)$ in a set of positive measure. This means that the basins of attraction are intermingled.

To prove that $\mathbb{T} \times \{ p_N,p_S\}$ is the likely limit set,
suppose it is not and consider the set $\Lambda$ of positive  measure of points whose $\omega$-limit
set is not contained in $\mathbb{T} \times \{ p_N,p_S\}$. By Fubini there is a
set $\mathbb{T} \times \{q\}$ that intersects $\Lambda$ in a set of positive measure.
%A positive iterate of $ U \times \{q\}$ will project surjectively to $[0,1]$.
Take a Lebesgue density point $(v,q)$ of $\Lambda \medcap \left(\mathbb{T} \times \{q\}\right)$.
We may take $q$ inside both the basin of $p_S$ and the unstable set of $p_N$ for $\varphi_t$.
Consider a small interval $J$ around $(v,q)$ in $\mathbb{T} \times \{q\}$.
There exists $n$ depending on $J$ so that $E^n (J)$ covers $\mathbb{T}$.

Note that $f^n_u (q)$, $u \in J$, are contained in the orbit of $q$.
Removing $\delta_0$-balls around $p_S$ and $p_N$, a compact part of the orbit of $q$ remains.
There is therefore a constant $c>0$ and a positive integer $m$ so that
$F^{n+m} (J)$ intersects either $W^s (\mathbb{T}\times \{p_N\})$
or $W^s (\mathbb{T}\times \{p_S\})$ in a set of measure at least $c$. Here $m$ and $c$ do not depend on $n$.
Since $E$ is a piecewise linear map, the proportion $\Lambda \medcap J$ in $J$
remains unchanged under iteration. Shrinking $J$ and  increasing $n$ makes that
$F^{n+m} (J) \medcap \Lambda$ goes to one.
In particular $F^{n+m} (J) \medcap \Lambda > 1-c$ for $n$ large,
which gives a contradiction.
\end{proof}

One can replace the base map $E$  by an invertible map such as a hyperbolic torus automorphism. We leave this to the reader, compare the discussion of Kan's example in \cite{MR2105774}.

\subsubsection{Multiple {metric} attractors} \label{subsub:mma}

We continue with a construction, in the vein of the setting of Theorem~\ref{t:ifsintermingledT}, of skew product systems with multiple {metric} attractors and mutually intermingled basins of attraction.
In the following $k\ge 4$ will be a positive even integer.
Let $\varphi_t$ be the flow of gradient Morse-Smale  vector field on the torus $\mathbb{T}^2$, with a sink $s_1$ with open and dense basin, and further equilibria $s_2,\ldots,s_k$ that are saddles or sources. It is easy to see that such vector fields exist, for any even $k\ge 4$.
For each integer $j$, $2 \le j \le k$, let $h_j : \mathbb{T}^2\to\mathbb{T}^2$  be a smooth
diffeomorphism that permutes equilibria by
\[
h_j (s_i) := s_{i + j - 1 \mod k+1}.
\]
For notational convenience, take $h_1$ to be the identity map.
Write $\psi_j: \mathbb{T}^2 \to \mathbb{T}^2$ for maps $h_j \circ \varphi_1 \circ h_j^{-1}$.
The maps $\psi_j$, $1 \le j \le k$ will then be conjugate to each other by a smooth conjugacy.
We find that $s_i$ is the unique attracting fixed point for $\psi_{i}$ with open and dense basin
$W^s (s_i)$.
Note also that the spectrum of $D\psi_i (s_{i+j \mod k+1})$ is equal to the spectrum of $D\psi_1 (s_j)$.

As before, let $E$ be an expanding map defined on $\mathbb{T}$ by $E (x) :=  L x \pmod 1$, where $L>1$ is a large enough natural number ensuring that $E$ has $k$ fixed points $q_1,\ldots, q_k$.
Let $J_1,\ldots, J_k$ be disjoint open intervals in $\mathbb{T}$, with $q_i \in J_i$. Let $t: \mathbb{T}\to [0,1]$ be a smooth function that is positive on $\medcup_{i=1}^k J_i$, vanishes outside $\medcup_{i=1}^k J_i$, and satisfies $t(q_i) = 1$.
Finally, for a parameter $u \in \mathbb{T}$, take $f_u: \mathbb{T}^2 \to \mathbb{T}^2$
to be a diffeomorphism on $\mathbb{T}^2$ with the following properties:
\begin{enumerate}

\item For $u \in J_i$ and $1 \le j \le k$,
\[f_u := h_j \circ \varphi_{t(u)} \circ h_j^{-1},\] \label{n:Milnor:conditions:2}  \vspace{0.1cm}

\item For each $i=1,\ldots,k$, \[\int_{\mathbb{T}} \log \|Df_u(s_i)\| du<0. \] \label{n:Milnor:conditions:3}

\end{enumerate}
Item~\eqref{n:Milnor:conditions:3} can be achieved by choosing the flow $\varphi_t$ so that the eigenvalues of $D\psi_{i} (s_i)$ (which equal those at $D\psi_1 (s_1)$) are
strongly contracting relative to the eigenvalues at the other equilibria.
Observe that $f_u$ is a diffeomorphism that depends smoothly on $u$. For $u$ outside $\medcup_{i=1}^k J_i$, $f_u$ is the identity map.
For each $u \in \medcup_{i=1}^k J_i$,
\begin{align*}% \label{n:Milnor:conditions:1}
\Omega (f_u) &= \{s_1,\ldots,s_{k}\}.
\end{align*}
For comparison and inspiration for further constructions we refer to \cite{MR4461435}.

Define the skew product $F$ on $\mathbb{T}\times \mathbb{T}^2$ by
\[
F(u,x) := (E(u),f_u(x)).
\]

\begin{theorem}\label{t:onS}
The mapping $F$ has $k$ {metric} attractors $\mathbb{T}\times \{s_i\}$, $i=1,\ldots,k$, whose basins are dense in $\mathbb{T}\times \mathbb{T}^2$ and mutually intermingled. The union $\mathbb{T} \times \{s_1,\ldots,s_{k}\}$ is the likely limit set.
\end{theorem}

\begin{proof}
By item (\ref{n:Milnor:conditions:3}) above, the measures $\lambda\times\delta_{s_i}$, $i=1,\ldots,k$, are invariant measures for $F$ supported on $\Lambda_i:= \mathbb{T}\times \{s_i\}$, $i=1,\ldots,k$ whose fiber Lyapunov exponents are negative. In particular, by Pesin theory, see also the proof of Theorem \ref{t:ifsintermingledT}, each $\Lambda_i$ is a {metric} attractor of $F$. To prove that their basins are intermingled,  suppose that $U\subset\mathbb{T}\times S$ is an open set. Choose a natural number $N$ such that
\begin{equation}\label{n:Milnore}
F^N(U)\medcap\left(\{q_i\}\times W^s(s_i)\right)\neq \emptyset.
\end{equation}
Any point in the intersection (\ref{n:Milnore}) tends to $\Lambda_i$ under the iteration of $F$.
The rest proceeds  as in the proof of Theorem~\ref{t:ifsintermingledT}.
\end{proof}

\subsubsection{Iterated function systems}

Reference \cite{MR3600645} contains elementary constructions of skew products of interval diffeomorphisms over shifts, arising from iterated function systems, admitting {metric} attractors with intermingled basins. Here we indicate analogous constructions for surface diffeomorphisms with
more attractors.

Consider the iterated function system generated by $\psi_i$, $1\le i\le k$ (defined in Section~\ref{subsub:mma}), and using equal probability $1/k$ for all of them.
Assume that for all $j$ between $1$ and $k$,
\begin{align}\label{e:ifsL<0}
\frac{1}{k}   \sum_{i=1}^{k}   \ln \|  D\psi_i (s_j)    \| &< 0.
\end{align}
Define
%$F:  \Sigma_{k} \times \mathbb{T}^2 \to  \Sigma_{k} \times \mathbb{T}^2$
$F:  \Sigma_{k} \times \mathbb{T}^2 \righttoleftarrow$
by
\begin{align}\label{e:ifsF}
F(\omega,x) &:= \left(\sigma \omega , \varphi_{\omega_0} (x)\right).
\end{align}
On $\Sigma_{k}$ we take Bernoulli measure $\nu$ corresponding to equal probabilities
for the symbols.
The inequality \eqref{e:ifsL<0} means that the fiber Lyapunov exponents at $s_j$, for any $1 \le j \le k$, are negative.  As in  Theorem~\ref{t:onS} we get the following result.

\begin{theorem}\label{t:ifsintermingledT2}
Consider the skew product system $F$ from \eqref{e:ifsF} on $\Sigma_{k} \times \mathbb{T}^2$. Take
the product of Bernoulli measure $\nu$ and Lebesgue measure $\lambda$ on $\mathbb{T}^2$ as reference measure.
The sets $\Sigma_{k} \times \{s_i\}$, $1 \le i \le k$ are minimal {metric} attractors
with intermingled basins. The union $\Sigma_{k} \times \{s_1,\ldots,s_{k}\}$ is the likely limit set.
\end{theorem}

\subsubsection{Skew product systems over minimal torus diffeomorphisms}

We provide a construction using a skew product system over minimal torus diffeomorphisms, motivated by \cite{MR2071239,MR1858557,MR2183303}.
Let $g$ be a conservative minimal $C^\infty$ diffeomorphism on $T^2$ with two ergodic measures $\mu_{\pm 1}$ each of which are absolutely continuous with respect to  Lebesgue measure. Such a diffeomorphism can be built using Anosov-Katok method from \cite{MR0370662} (see \cite{MR1858557,MR2183303} for details). Put
$$
\Phi := \left\{\varphi:\mathbb{T}^2\to\mathbb{R} \; ; \; \int_{\mathbb{T}^2} \phi\, d\mu_{-1}<0<\int_{\mathbb{T}^2} \varphi\, d\mu_1\right\}.
$$
It is not difficult to see that $\Phi$ is nonempty.

\begin{lemma}\label{Definign constant}
For any $\varphi\in\Phi$, there is a constant $c\in (0,1)$ such that the mapping
$H_c:\mathbb{T}^2\times \mathbb{T} \righttoleftarrow$
% $H_c:\mathbb{T}^2\times \mathbb{T} \to  \mathbb{T}^2\times \mathbb{T}$
defined by
\begin{equation*}
H_c(x,t) := (g(x),\varphi(x) + c + t  \pmod 1)
\end{equation*}
is minimal.
\end{lemma}

\begin{proof}
First note that
\begin{equation}\label{e:H^n_c}
H^n_c (x , t)=\left(g^n(x),  \varphi^n (x) + n c   + t  \pmod 1 \right),
\end{equation}
where by slight abuse of notation we have written
\[
\varphi^n (x) := \sum_{i=0}^{n-1} \varphi (g^i(x)).
\]
Let $\{V_k\}_{k\in\mathbb{N}}$ be an open basis of the topology on $\mathbb{T}^2$ and $x\in \mathbb{T}^2$ an arbitrary point. For any $k\in\mathbb{N}$, choose a sequence $\{n_i^{(k)}\}$ of natural numbers such that $g^{n_i^{(k)}}(x)\in V_k$. Now, consider the mapping $\Theta_{(x,k)}$ from $[0,1]$ to the set of all closed subsets of $[0,1]$ (endowed with the Hausdorff distance) defined by
$$
\Theta_{(x,k)}(c) :=   \text{Cl}\left(\left\{  \varphi^{n_i^{(k)}} (x)+n_i^{(k)} c \pmod 1 \right\}_{i=0}^\infty\right).
$$

\medskip
\noindent{\bf Claim.} The mapping $\Theta_{(x,k)}$ is lower semi-continuous.
\begin{proof}
For any given $\varepsilon>0$, choose $i_0\in\mathbb{N}$ such that
$$
\text{Cl}\left(\left\{\varphi^{n_i^{(k)}}(x)+n_i^{(k)} c \pmod 1\right\}_{i=0}^\infty \right) \subseteq
U_{\varepsilon/2} \left( \left\{\varphi^{n_i^{(k)}}(x)+n_i^{(k)} c \pmod 1\right\}_{i=0}^{i_0}\right).
$$
Here $U_{\varepsilon/2}(\cdot)$ stands for the open $\varepsilon/2$-neighborhood of the given closed set.
Now, if $\tilde{c}$ is sufficiently close to $c$ then
\begin{align*}
U_\varepsilon\left(\text{Cl}\left( \left\{\varphi^{n_i^{(k)}}(x)+n_i^{(k)} \tilde{c} \pmod 1\right\}_{i=0}^\infty\right)\right) \supseteq &
~U_\varepsilon\left(\left\{\varphi^{n_i^{(k)}}(x)+n_i^{(k)} \tilde{c} \pmod 1\right\}_{i=0}^{i_0}\right)
\\
\supseteq &~  U_{\varepsilon/2}\left(\left\{\varphi^{n_i^{(k)}}(x)+n_i^{(k)} c~(\text{mod}\,1)\right\}_{i=0}^{i_0}\right)
\\
\supseteq & ~ \text{Cl}\left(\left\{\varphi^{n_i^{(k)}}(x)+n_i^{(k)} c \pmod 1\right\}_{i=0}^\infty\right).
\end{align*}
\end{proof}

Let $\mathcal{R}_{(x,k)}$ be the set of continuity points of $\Theta_{(x,k)}$, which is known to be a residual set, and put $\mathcal{R}_x=\medcap_k \mathcal{R}_{(x,k)}$, a residual set again.
%This set is nonempty \cite{MR0046636}.
We prove that for any $c\in\mathcal{R}_x$, $\left\{H^n_c(x,t)\right\}_{n\in\mathbb{N}}$ is dense in $\mathbb{T}^2\times \mathbb{T}$. To prove it, let $U\times I$ be an arbitrary open set in $\mathbb{T}^2\times\mathbb{T}$ and choose $k$ such that $V_k\subset U$. Suppose by contradiction,
\begin{equation}\label{Mapping H}
H_c^{n_i^{(k)}}(x,t)\not\in U\times I,~ \text{for any}~ i.
\end{equation}
On the other hand, recall \eqref{e:H^n_c},
\begin{equation}\label{e:Hni}
H_c^{n_i^{(k)}}(x,t)=\left(g^{n_i^{(k)}}(x),\varphi^{n_i^{(k)}} (x)+n_i^{(k)}c+t \pmod 1 \right).
\end{equation}
As $g^{n_i^{(k)}}(x)\in V_k$, \eqref{Mapping H} implies that $\varphi^{n_i^{(k)}} (x)+n_i^{(k)}c+t \pmod 1\not\in I$, for any $i$. That is, $\Theta_{(x,k)}(c)\medcap I-t = \emptyset$. However, \eqref{e:Hni} makes clear that there are arbitrarily small $\delta>0$ so that
$$
\Theta_{(x,k)}(c+\delta) \medcap  I-t \neq \emptyset.
$$
This contradicts the fact that $c$ is a continuity point of $\Theta_{(x,k)}$.

Having established the existence of a point $(x,t)$ with dense positive orbit, proving minimality follows from a standard argument using expression \eqref{e:H^n_c} and minimality of $g$.
%  (compare \cite[Exercise~2.2.6]{MR3558919}).
Namely, for any $\varepsilon>0$, there is $N_0$ so that $\left\{(x,t), H_c(x,t),\ldots,H^{N_0}_c (x,t)\right\}$
is $\varepsilon$-dense in $\mathbb{T}^2\times \mathbb{T}$. By expression  \eqref{e:H^n_c} the same is true for any point from $\{x\}\times \mathbb{T}$ replacing $(x,t)$. For an arbitrary point $(y,s) \in \mathbb{T}^2 \times \mathbb{T}$, since $g$ is minimal, its positive orbit accumulates on a point
$(x,u)$. It is therefore $\varepsilon$-dense. Since $\varepsilon$ is arbitrary,
$\left\{H^n_c (y,u)\right\}_{n\in\mathbb{N}}$ is dense for any $(y,s) \in \mathbb{T}^2\times\mathbb{T}$.
\end{proof}

Let $F$ be a north pole ($p_N$)/south pole ($p_S$) diffeomorphism on $\mathbb{T}$ and consider its suspension flow defined by $F_s(y,t)=(y,t+s)$ on the suspended manifold $M$ given by $(y,t+1)\sim(F(y),t)$ (see for instance  \cite[Chapter~3, Proposition~3.7]{MR0669541} and \cite[Section~3.4.1]{MR3558990}).
The flow $F_s$ has invariant sets
\[
\Lambda_S:= \{p_S\}\times \mathbb{T}, \qquad \Lambda_N:= \{p_N\}\times \mathbb{T}.
\]
Choose $\varphi\in\Phi$ and a constant $c$ such that the mapping $H_c$ defined in Lemma \ref{Definign constant} is minimal. For simplicity denote $\varphi+c$ by $\varphi$ again. Define % % $H: \mathbb{T}^2\times M\to \mathbb{T}^2\times M$
$H: \mathbb{T}^2\times M\righttoleftarrow$
by
\begin{equation}\label{Defining H}
H\left(x,(y,t)\right):=\left(g(x), F_{\varphi(x)}(y,t)\right).
\end{equation}

\begin{theorem}
The mapping $H$ defining by (\ref{Defining H}) has two metric attractors $\mathbb{T}^2\times \Lambda_S$ and $\mathbb{T}^2\times \Lambda_N$ whose basins are intermingled.
\end{theorem}

\begin{proof}
Note that
\begin{align}\label{Np/Sp:AK method}
H^n\left(x,(y,t)\right)
&=
\left(g^n(x), F_{\varphi^n(x)}(y,t)\right) \nonumber
\\
&=
\left(g^n(x), F^{[t+\varphi^n(x)]}(y), t+\varphi^n(x) \pmod  1 \right),
\end{align}
where $[\cdot]$ denotes the integer part and as before, $\varphi^n(x)=\sum_{i=0}^{n-1} \varphi (g^i(x))$.

In view of Lemma \ref{Definign constant}, $\mathbb{T}^2\times \Lambda_S$ and $\mathbb{T}^2\times \Lambda_N$ are metric attractors of $H$. We prove that the basins of attraction are intermingled.
For this, let $\big(x,(y,t)\big)$ be an arbitrary point in $\mathbb{T}^2\times M$. Take $\tilde{x}$ close to $x$ to be a generic point with respect to $\mu_1$ and $\tilde{y}\neq p_N$ close to $y$. It is not difficult to see that by \eqref{Np/Sp:AK method},
$$
H^n\left(\tilde{x},(\tilde{y},t)\right)\to \mathbb{T}^2\times \Lambda_S
$$
as $n \to \infty$.
Likewise take $\hat{x}$, a generic point with respect to $\mu_{-1}$, and $\hat{y}\neq p_S$, sufficiently close to $x$ and $y$ respectively. Then
$$
H^n\left(\hat{x},(\hat{y},t)\right)\to \mathbb{T}^2\times \Lambda_N
$$
as $n\to \infty$.
\end{proof}

\section{Flows of smooth vector fields}\label{s:examplesforflows-p}

We construct similar examples for flows of smooth vector fields.
Let us first remark that the definitions of {metric} attractor, likely limit set, and intermingled basin given in the introduction for continuous maps transfer to flows $x\mapsto \phi_t (x)$ generated by differential equations.

We will consider skew product systems over an ergodic volume preserving flow.
To be concrete, let $g$ be a smooth vector field that generates the suspension
%(see for instance  \cite[Chapter~3, Proposition~3.7]{MR0669541} and \cite[Section~3.4.1]{MR3558990})
of a hyperbolic torus automorphism.
The flow $\psi_t: M \to M$ of
\begin{align}
\label{e:g-p}
\dot{u} &=  g(u)
\end{align}
on $M := \mathbb{T}^3 = \left(\mathbb{R}/\mathbb{Z}\right)^3$ preserves Lebesgue measure $\lambda$ and is ergodic.
Consider the flow $\varphi_t: N \to N$ of a Morse-Smale gradient differential equation
$\dot{x} = f(x)  = -\text{grad}(h)$ on a compact manifold $N$ as in \eqref{e:odef-p}, \eqref{e:odegradh}.
We will also write $\varphi(t,x) = \varphi_t (x)$.
Take a skew product systems of the form
\begin{align}
\nonumber
\dot{u} &=  g(u),
\\
\label{e:zetagf-p}
\dot{x} &= \zeta(u,x) f(x),
\end{align}
with $(u,x) \in M \times N$ and with $\zeta : M \times N \to \mathbb{R}$ a smooth scalar function.
Take volume on $M\times N$ as reference measure.
Write $\Phi_t(u,x)$ for the flow of \eqref{e:zetagf-p}.

\begin{example}\label{ex:eta-p}
A special case is where $\zeta(u,x)  =\eta(u)$ depends only on $u$.
Let $\eta : M \to \mathbb{R}$ be a smooth scalar function taking positive and negative values
and satisfying
\begin{align}
\label{e:eta>0-p}
\int_{M} \eta (u) \, d\lambda(u) &> 0.
\end{align}
Consider
\begin{align}
\nonumber
\dot{u} &=  g(u),
\\
\label{e:etagf-p}
\dot{x} &= \eta(u) f(x).
\end{align}
For different values of $\eta$, the flow $\varphi(t,x_0)$ of $f$ is followed in different time directions.
Observe that
\[
\Phi_t (u_0,x_0) =  \left( \psi_t (u_0) , \varphi (\tau(t) , x_0) \right)
\]
with
\[
\tau(t) = \int_0^t \eta(u(s)) \, ds.
\]
By Birkhoff's ergodic  theorem we find
$\tau(t) \to \infty$ as $t \to \infty$ for almost all $u_0 \in M$.
Thus $M \times \{p_S\}$ is a {metric} attractor and also the likely limit set.
$\hfill \blacksquare$
\end{example}

Analogous to the choice of the function $s$ in Section~\ref{s:circle}, take the smooth function $\zeta :  M  \times N \to \mathbb{R}$ with the following properties:
\begin{enumerate}

\item There are  hyperbolic periodic orbits $q_1, q_2$ for $g$ for which
\label{i:0}
\[
\left. \zeta \right\vert_{\{q_1\} \times N} > 0, \qquad \left. \zeta \right\vert_{\{q_2\} \times N} < 0.
\]

\item We have
\[
\int_{M}  \zeta (u,p_S) \, d\lambda(u) > 0, \qquad \int_{M}  \zeta (u,p_N) \, d\lambda(u) < 0.
\]
 \label{i:3}

\end{enumerate}
Now
\[\lim_{t\to \infty} \frac{1}{t} \ln \left(\left\| D\varphi_{\tau(t)}  (p_S) \right\|\right), \qquad \text{with }\tau (t) = \int_0^t \zeta (u(s), p_S)\, ds,
\]
is negative, which
means as in Section~\ref{s:circle} that the fiber Lyapunov exponents at $p_S$ are negative.
Likewise we have only negative fiber Lyapunov exponents at $p_N$.

\begin{theorem}\label{t:flowintermingled-p}
Consider the differential equation  \eqref{e:zetagf-p} on $M \times N$
and suppose $\zeta$ satisfies \eqref{i:0}, \eqref{i:3} above.
Take volume  on $M \times N$ as reference measure.
Then $M \times \{ p_N\} $ and  $M \times \{p_S\}$
are {metric} attractors for the flow $\Phi_t$ of \eqref{e:zetagf-p} with intermingled basins.
The union $M \times \{ p_N,p_S\}$ is the likely limit set.
\end{theorem}

\begin{proof}
Take a neighborhood $U$ of $p_S$ and a smooth scalar  function $\eta : M \to \mathbb{R}$ so that
\begin{align}
\label{e:eta<zeta-p}
\eta (u) < \zeta (u,x), \qquad u \in M, x \in U,
\end{align}
and so that \eqref{e:eta>0-p} holds.
Let $z(t)$ be a solution to $\dot{z} = f(z)$ contained in the basin of attraction of $p_S$ for the flow $\varphi_t$.
Take a solution $(u(t) , x(t))$ to \eqref{e:etagf-p} with $u(0) = u_0$ and $x(0) = z (T)$ for some $T$. Write $x(t) = z(\tau(t))$, where $\tau (0)=T$.

For almost all $u_0 \in M$,  $\lim_{t\to\infty} \tau (t) = \infty$ (see Example~\ref{ex:eta-p}). For such $u_0$, $\tau (t)$ for $t \ge 0$  has a minimum value.
For a larger starting point $\tau(0) = T$ the minimum value increases.
Now $T$ large corresponds to
an initial point close to $p_S$. So for $T$ large enough,   $x(t)$, $t \ge 0$,  stays in a neighborhood of $p_S$ and converges to $p_S$.
This shows that for almost all $u \in M$,
there are local stable manifolds $\{u\} \times U_u$ of $(u,p_S)$.
By \eqref{e:eta<zeta-p},
these manifolds are local stable manifolds of $(u,p_S)$ also for \eqref{e:zetagf-p}.
Likewise there are local stable manifolds $u \times  V_u$ of $(u,p_N)$, for
which
$\Phi_t (u,z)$, $z \in V_u$, converges to $M \times \{p_N\}$ as $t \to \infty$.
We conclude that  $M \times \{p_N\}$ and $M \times \{p_S\}$ are {metric} attractors.

Write
\[
 W^s (M\times \{p_S\}) :=   \bigcup_{u \in M} \{u\} \times U_u
\]
 for the union of local stable manifolds of $(u,p_S)$, and
\[
 W^s (M\times \{p_N\}) :=   \bigcup_{u \in M} \{u\} \times V_u
\]
 for the union of local stable manifolds of $(u,p_N)$.
Now take an open set $U \subset M \times N$.
Since stable manifolds of periodic orbits of \eqref{e:g-p} lie dense in $M$,
any open set will intersect the stable manifolds of $q_1$ and $q_2$ from \eqref{i:0}.
This implies that for large enough $t$, $\Phi_t (U)$ intersects both  $W^s (M\times \{p_S\})$  and  $W^s (M\times \{p_N\})$ in sets of positive measure.
The {metric} attractors $M \times \{p_N\}$ and $M \times \{p_S\}$
therefore have intermingled basins of attraction, where both basins lie dense in $M\times N$.

The proof that $M \times \{p_N,p_S\}$ is the likely limit set goes as in
Theorem~\ref{t:ifsintermingledT}. We give a brief account.
As in that proof, consider the set $\Lambda\subset M\times B$ consisting of points whose $\omega$-limit sets are not contained in
$M \times \{p_N,p_S\}$.   Suppose $\Lambda$ has positive volume.
Then by Fubini's theorem there is a slice $M \times \{y\}$ for some $y \in B$
that intersects $\Lambda$ in a set of positive volume.
Take a Lebesgue density point $(v,y)$ of $\Lambda \medcap \left(M \times \{y\}\right)$.
Consider the flow of a small ball $\Sigma$ around $(v,y)$ in $M \times \{y\}$. For high enough $t$,  $\Phi_t (\Sigma)$ will intersect the union of local stable manifolds of
$(u,p_N)$, and  the union of local stable manifolds of
$(u,p_S)$. These intersections are of positive measure.
Since the dynamics in the base is a suspension of a linear hyperbolic torus map, the proportion $\Lambda \medcap \Sigma$ in $\Sigma$
remains unchanged under iteration. This gives a contradiction.
\end{proof}

\begin{corollary}
Consider the time-one map $\Phi_1$ on $\mathbb{T}^2 \times N$.  Then $\mathbb{T}^2 \times \{ p_N\}$ and $\mathbb{T}^2 \times \{p_S\}$ are {metric} attractors for $\Phi_1$ with intermingled basins.
\end{corollary}

\begin{proof}
Take $(u_0,0) \in \mathbb{T}^2 \times \mathbb{T}$.
% Consider $F = \Phi_1$ acting on $\mathbb{T}^2 \times N$.
The orbit  $\Phi_t (u_0,0, x_0)$ converges to $M \times \{p_S\}$ if and only if the sequence
$(u_n,x_n) =  \Phi_n (u_0,x_0)$, $n \in \mathbb{N}$,  converges to $\mathbb{T}^2 \times \{p_S\}$.
\end{proof}

\section{Intermingled basins of two thick attractors}

This section contains a construction of a skew product system, arising from an iterated function system, possessing a pair of thick attractors with intermingled basins.
We start with the construction of an iterated function system with both a thick attractor and a thick repeller, after which a random walk on orbits is introduced. We continue with sketching alternative constructions, leading to skew product systems with several thick metric attractors
and mutually intermingled basins. We will restrict to skew product systems over symbolic dynamics, and will not discuss extensions to smooth maps.

\subsection{Iterated function system with a thick attractor/thick repeller pair}

In building a model we start with the introduction of an iterated function system on the unit interval $[0,1]$
generated by two diffeomorphisms $f_0,f_1$.
\begin{figure}[!ht]
	\begin{center}
        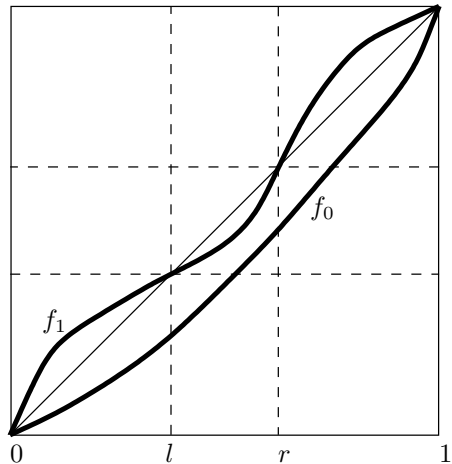
		\caption{\label{f:thick-attractor-repeller} We consider an iterated function system generated by diffeomorphisms $f_0$ and $f_1$ on $[0,1]$ with graphs as depicted. There is an invariant interval $I_l = [0,l]$. For the inverse maps $f_0^{-1}$ and $f_1^{-1}$,   the interval $I_r = [r,1]$ is invariant.  }
\end{center}
\end{figure}
The graphs of $f_0$  and $f_1$ are sketched in Figure~\ref{f:thick-attractor-repeller}.
Both maps fix the end points $0$ and $1$.
For points $0<l<r<1$ we have
\begin{align*}
f_1 (l) &= l, \hspace{1cm} f_1(x) > x \textrm{ for } x \in (0,l),
\\
f_1 (r) &= r, \hspace{1cm} f_1(x) > x \textrm{ for } x \in (r,1).
\end{align*}
We further have
\begin{align*}
f_0(x) < x \textrm{ for } x \in (0,1), \qquad
f_1(x) < x \textrm{ for } x \in (l,r).
\end{align*}
The interval $I_l = [0,l]$ is mapped into itself by $f_0$ and $f_1$.
The interval $I_r = [r,1]$ is mapped into itself by $f_0^{-1}$ and $f_1^{-1}$.
We take conditions on derivatives at $0$ and $1$, namely
\begin{align*}%\label{e:f'}
f_0'(0)f_1'(0)  > 1,
\qquad
(f_0^{-1})'(1)(f_1^{-1})'(1)  > 1,
\end{align*}
and generic conditions on second order derivatives, see \cite[Section~3.1]{MR2853609} or \cite[Proposition~2.1]{MR3600645}.

The iterated function system generated by $f_0$ and $f_1$ has a representation as a skew product system
%$F: \Sigma_2 \times [0,1] \to \Sigma_2 \times [0,1]$
$F: \Sigma_2 \times [0,1] \righttoleftarrow$
given by
\begin{align*}%\label{e:Fifs}
F(\omega , x) &:= \left(\sigma \omega , f_{\omega_0} (x) \right).
\end{align*}
This defines a homeomorphism on $ \Sigma_2 \times [0,1]$.
Denote iterates of $F$ by
\[
F^n (\omega,x) = (\sigma^n , f^n_\omega (x)).
\]
On $\Sigma_2$ we take Bernoulli measure $\nu_{1/2}$.
By \cite{MR2644340} the map $F$ restricted to
$\Sigma_2 \times I_l$ admits a thick {metric} attractor
\[
\Lambda_l := \bigcap_{n=0}^{\infty} F^n ( \Sigma_2 \times I_l )
\]
with $\nu_{1/2}\times \lambda  (\Lambda_l) > 0$.
The attractor is characterized by
\begin{align*}
%\label{e:Lambdal}
\Lambda_l &= \left\{  (\omega , [0, X_l (\omega)] ) \; ; \; \omega \in \Sigma_2 \right\}
\end{align*}
for an invariant measurable function $X_l : \Sigma_2 \to [0,l]$
with $X_l > 0$ almost everywhere.
The values $X_l(\omega)$ are obtained from a pullback construction
\begin{align}
\label{e:Xlpullback}
X_l(\omega) &:= \lim_{n\to\infty} f^n_{\sigma^{-n} \omega} (l),
\end{align}
where $f^n_{\sigma^{-n} \omega} (l)$ is a monotone decreasing sequence.
For any $(\omega,x) \in \Lambda_l$ with $0 < x < X_l(\omega)$ and any neighborhood $U$ of $(\omega,x)$ there is a point $(\omega', X_l (\omega')) \in U$.

The skew product $F$ likewise admits a thick metric repeller
\[
\Lambda_r := \bigcap_{i=0}^{\infty} F^{-i} ( \Sigma_2 \times I_r )
\]
 in $\Omega^+ \times I_r$. That is,
$\nu_{1/2} \times \lambda (\Lambda_r) >0$ and $\Lambda_r$ is a thick attractor for $F^{-1}$.
The argument of \cite[Remark~1]{MR2644340} and \eqref{e:Xlpullback} shows that every open ball $U \subset \Sigma_2 \times [0,1]$ has nonempty intersection with the complement of $\Lambda_l \medcup \Lambda_r$.

\subsection{Thick {metric} attractors with intermingled basins}\label{subs:thick}

The next step is to define a random walk on orbits of $F$.
As before, With $\Omega := \{-1,+1\}^\mathbb{Z}$, define
%$G: \Omega \times  \Sigma_2 \times [0,1] \to \Omega \times  \Sigma_2 \times [0,1]$
$G: \Omega \times  \Sigma_2 \times [0,1] \righttoleftarrow$
by
\begin{align*}
%\label{e:G4}
G(\eta,\omega,x) &:= \left(\sigma \eta , F^{\eta_0} (\omega,x) \right).
\end{align*}
Take a continuous function $p: [0,1] \to (0,1)$ with
\[
p(x) = p_l > 1/2, \text{ for } x \in I_l, \quad p(x) = p_r < 1/2, \text{ for }  x \in I_r.
\]
Let $\zeta_{\omega,x}$ be the measure on $\Omega$ which is defined on cylinders by
\begin{align*}
% \label{e:nuomegax}
\zeta_{\omega,x} ([a_0\cdots a_k]) &:=  \prod_{i=0}^k p_{a_i} (f^{S_i (a)}_\omega (x)),
\end{align*}
where, as before, $p_{-1} (x) := 1 - p (x)$ and $p_{1} (x) := p(x)$ and
$S_i(a) = \sum_{i=0}^{i-1} a_i$.
Define the measure $\mu$ on $\Omega \times \Sigma_2 \times [0,1]$ by
\begin{align*}%\label{e:mu2}
\mu (A) &:= \int  \mathbbm{1}_{A_{(\omega,x)}} \,  d\zeta_{\omega,x}    d (\nu_{1/2}\times \lambda) (\omega,x),
\end{align*}
where $A_{(\omega,x)} = A \medcap (\Omega^+ \times \{(\omega,x)\})$.
Iterates of $G$ are of the form
\begin{align*}
G^n (\eta,\omega,x) &=  \left(\sigma^n \eta , \sigma^{S_n(\eta)} \omega ,  f^{S_n (\eta)}_\omega (x) \right). % \text{ with } S_n(\eta) = \sum_{i=0}^{n-1} \eta_i.
\end{align*}
As in the proof of Theorem~\ref{t:northsouthintermingled}, we find that for $\mu$-almost all $(\eta,x)$, $S_n(\eta)$ goes to $\infty$ or $-\infty$ as $n\to\infty$.

With
%$\chi: \Omega\times\Sigma_2 \to \Omega\times\Sigma_2$
$\chi: \Omega\times\Sigma_2 \righttoleftarrow$
given by
\[
\chi (\eta,\omega) := (\sigma \eta, \sigma^{\eta_0} \omega),
\]
we can consider $G$ as a skew product of interval diffeomorphisms over $\chi$.
Namely,
 \[
G(\eta,\omega,x) = \left\{  \begin{array}{ll}
                          \left( \chi (\eta,\omega), f_0(x)\right), & (\eta_0,\omega_0) =  (+,0), \vspace{0.1cm}
                          \\
  \left(\chi (\eta,\omega),f_1(x)\right), &  (\eta_0,\omega_0) =  (+,1),\vspace{0.1cm}
\\
  \left(\chi (\eta,\omega),f_0^{-1}(x)\right), &  (\eta_0,\omega_{-1}) =  (-,0), \vspace{0.1cm}
\\
  \left(\chi (\eta,\omega),f_1^{-1}(x)\right), &  (\eta_0,\omega_{-1}) =  (-,1).
\end{array}
\right.
\]
For (positive) orbits contained in $\Omega\times\Sigma_2 \times I_l$, the measure we consider on $\Omega$ is $\nu_{p_l}$. For such orbits  the following lemma becomes useful.

\begin{lemma}
The measure $P := \nu_{p_l} \times \nu_{1/2}$  is an invariant ergodic measure for $\chi$.
\end{lemma}

\begin{proof}
Invariance of $P$ is clear and so we only have to prove ergodicity.
In the same way as the shift on $\Sigma_2$ endowed with Bernoulli measure $\nu_{p}$ is measurably isomorphic to the two-dimensional baker map $B_p$ (from \eqref{e:baker}), one has that
$\chi$, endowed with $P$, is measurably isomorphic to
% $J: [0,1]^2 \times [0,1]^2 \to     [0,1]^2 \times [0,1]^2 $
$J: [0,1]^2 \times [0,1]^2 \righttoleftarrow$
given by
\[
J(x,y,u,v) := \left\{
                        \begin{array}{ll} \left(B_p(x,y) ,  B_{1/2} (u,v)\right),  &  0 \le x < p, \vspace{0.1cm}
                                             \\
                                          \left(B_p(x,y) , B_{1/2}^{-1} (u,v)\right),  &  p \le x \le 1,
                        \end{array}
             \right.
\]
and endowed with Lebesgue measure.
Lebesgue measure is indeed invariant for $J$. For Lebesgue almost all points, there are local stable and local unstable manifolds.
Namely, for Lebesgue almost all $x_0$, $W^s  := \{ x = x_0, 0\le y,u,v < 1\}$ has the property that $G^n (x_0,y,u,v) - G^n(x_0,y',u',v')$ goes to zero as $n\to \infty$.
Similar for $W^u  := \{ y = y_0, u=u_0, v=v_0, 0\le x < 1\}$ and iterates under $G^{-n}$.
A standard Hopf argument (see \cite[Section~4.2.6]{MR3558990}) shows ergodicity, just as for hyperbolic diffeomorphisms. We note that 
\cite[Theorem~3.2]{homburg2022iterated}
% \cite[Theorem~3.2]{MR4456120}  
treats a similar situation of piecewise linear maps.
\end{proof}

\begin{theorem}
Consider the skew product system $F$ from \eqref{e:F} on  $\Omega\times\Sigma_2 \times [0,1]$. Take $\mu$ as reference measure.
The sets  $\Omega\times \Lambda_l$ and $\Omega\times \Lambda_r$
are thick {metric} attractors for $G$ with intermingled basins.
The union  $\Omega\times \left(\Lambda_l \medcup \Lambda_r\right)$ is the likely limit set.
\end{theorem}

\begin{proof}
The identity
\[
G \left( \eta , \omega, X_l (\omega)  \right) =
\left( \sigma \eta , \sigma^{\eta_0} \omega , X_l (\sigma^{\eta_0} \omega) \right)
\]
implies that $\Omega \times \Lambda_l$ is invariant for $G$.
The same is true for $\Omega \times \Lambda_r$.
As in the proof of Theorem~\ref{t:northsouthintermingled} we find that $\mu$-almost all orbits
converge to either $\Omega\times \Lambda_l$ or $\Omega\times\Lambda_r$. It remains to show that these sets are thick {metric} attractors.

To show that $\Omega\times \Lambda_l$ is a thick {metric} attractor we follow  \cite[Lemma~2]{MR2644340}.
Let us summarize this approach for the skew product system $F$ on $\Sigma_2 \times I_l$ over the shift $\sigma$, and then discuss how to adapt it to the setting of a skew product over $\chi$.

An invariant measure for $F$ is called good if it a weak limit of $\frac{1}{n} \sum_{k=0}^{n-1} F^k_* (\nu_{1/2} \times \lambda)$. Put
$$A_{\text{min}}:=\text{Cl}\big( \bigcup_{\mu\,\text{is a good measure}} \text{Supp}(\mu)\big).$$
Following the notation of \cite{MR2644340}, we call $A_{\text{min}}$ the minimal attractor of $F$. By \cite{MR1135904,MR1409419}, $A_{\text{min}}\subseteq\Lambda_r$ and furthermore,  as shown in \cite{MR2644340},
\[
\nu_{1/2} \left(  \left\{ \omega \in \Sigma_2\; ; \; (\omega,X_l(\omega)) \in A_{\text{min}}  \right\} \right) = 1,
\]
after which the proof can be concluded by studying properties of $X_l$.

Now, we turn to the skew product $G$. Put
\[
E  := \left\{ \eta \in \Omega \; ; \; \sum_{i=0}^{\infty} \eta_i   = \infty\right\}
\]
and note that $\nu_{p_l} (E)=1$.
For a point $p \in \Omega\times\Sigma_2\times [0,1]$, write $x(p)$ for the coordinate in $[0,1]$.
For $\eta \in E$ and $\omega\in\Sigma_2$  there exists $y := y(\eta,\omega) \in (X_l(\omega),l)$ so that
$x (G^n (\eta,\omega,u)) \le l$ for all $u \le y$ and $n\ge 0$. In fact
$G^n (\eta,\omega,y))$ converges to the graph of $X_l$ as $n\to \infty$.
The set
\[
\mathcal{K} := \{ (\eta,\omega,u) \in \Omega\times \Sigma_2\times I_l \; ; \; \eta \in E,  X_l(\omega) < u < y(\eta,\omega) \}
\]
forms a set of positive measure $\mu(\mathcal{K}) >0$.
Since the $\text{graph}(X_l)$ is $G$-invariant, $G^k_*(\mu|_\mathcal{K})(\mathcal{K})=\mu(\mathcal {K})>0$, for any $k$. On the other hand, $G^k(\mathcal{K})$ tends to $\text{graph}(X_l)$. If $\mu_\infty$ is a weak limit point of the sequence $\mu_n := \frac{1}{n} \sum_{k=0}^{n-1}  G^k_* \left(\left.\mu\right\vert_{\mathcal{K}}\right)$, a good measure, then we get
\[
\mu_\infty (\text{graph}\, (X_l) )=\lim_{n\to\infty}\mu_n (\text{graph}\, (X_l) )\ge \mu (\mathcal{K} ) > 0. \]

Hence (denoting the minimal attractor of $G$ by $A_{min}$)
\[
P \left\{ (\eta,\omega) \in \Omega\times \Sigma_2 \; ; \;  (\eta,\omega,X_l (\omega)) \in A_{min}  \right\} > 0.
\]
Ergodicity of $\chi$ implies that this measure is $1$:
\[
P\left(  \left\{ (\eta,\omega) \in \Omega\times\Sigma_2\; ; \; (\eta,\omega,X_l(\omega)) \in A_{min}  \right\} \right) = 1.
\]

The reasoning
of \cite[Lemma~4]{MR2644340} can be followed to conclude that the {metric} attractor
equals $\Omega \times (\Lambda_l \medcup \Lambda_r)$.
This can be concluded from the next property which we claim to hold:
for any $(\eta,\omega,x) \in \Omega\times \Sigma_2\times I_l$ with $0 < x < X_l(\omega)$ and any neighborhood $U$ of $(\eta,\omega,x)$ there is set $S \subset \Omega \times \Sigma_2$ with $P(S) >0$ and  $(\eta',\omega', X_l (\omega')) \in U$ for $(\eta',\omega') \in S$.

To see this, take $(\hat{\eta},\omega,\hat{x}) \in U$ with $\hat{\eta} \in E$. Take an open neighborhood $V \subset U$ of $(\hat{\eta},\omega,\hat{x})$.
We may assume $V$ is of the form $C_\Omega \times C_\Sigma \times J$ for cylinders $C_\Omega = [\hat{\eta}_{-m}, \ldots, \hat{\eta}_m]$ , $C_\Sigma = [\omega_{-n}, \ldots, \omega_n]$ and an open interval $J$.
For $n$ given and large we can take $m\ge n$ so that  $\chi^{-m} (\hat{\eta} , \omega) =  (\sigma^{-m} \hat{\eta} , \sigma^{-n} \omega)$. We can moreover take $m$ so that $\chi^{-i}(\hat{\eta} , \omega)$, $-m \le i \le 0$,  only involves $\omega_j$ with $-n \le j$.
Let $J_n$ be such that
\[
f^{n}_{ \sigma^{-n} \omega } (J_n) = J.
\]

We will find points $(\eta',\omega', X_l (\omega')) \in V$.
Consider the distribution function
\[
\Phi (x) = P(  \{ (\eta,\omega) \in \Omega\times \Sigma_2 \; ; \; X_l (\omega) \le x \})
= \nu_{1/2} (  \{ \omega \in \Sigma_2 \; ; \; X_l (\omega) \le x \}).
\]
As this is a monotone function, its derivative $\psi$ exists almost everywhere. Let $x_0$ be such that $\psi(x_0)>0$. Then $P(S_\delta) >0$ for $S_\delta = \{ (\eta,\omega) \; ; \; |X_l (\omega)-x_0| \le \delta \}$.
By \cite{MR2644340} there exists a word $\beta$ with $f^k_\beta (x_0) \in J_n$.
Let $\omega'_i = \omega_i$ for $-n \le i \le n$ and $\omega'_{-n+i} = \beta_i$ for $-k \le i \le -1$. Let $\eta'_i = \hat{\eta}_i$ for $-m \le i \le m$ and $\eta'_{-m+i} = `"+1"$ for $-k \le i \le -1$.
We have
$x (G^{m+k} (\chi^{-(m+k)} (\eta',\omega') , x_0)) \in J$.
For $(\bar{\eta},\bar{\omega},x)$ with $x = X_l (\bar{\omega})$ within distance $\delta$ of $x_0$ concatenate
$(\tilde{\eta},\omega)$ on the left with the negative part of $(\bar{\eta},\bar{\omega})$:
$\eta'_{-m+k+i} = \bar{\eta}_{i}$ for $i < 0$ and
$\omega'_{-n+k+i} = \bar{\omega}_{i}$ for $i < 0$.
This provides the required positive measure set of sequences.
\end{proof}

\subsection{Multiple thick metric attractors with intermingled basins}

We discuss a slightly different approach and a possible extension leading to examples
with multiple thick metric attractors. We will only sketch the constructions and will not provide complete arguments.

\subsubsection{An alternative construction}\label{One Dimensional}

% We sketch a slightly different approach.
Take two diffeomorphisms $f_0, f_1$ on $[0,1]$ with graphs as depicted in the left panel of Figure~\ref{f:thick-attractor-attractor}.
\begin{figure}[!ht]
	\begin{center}
		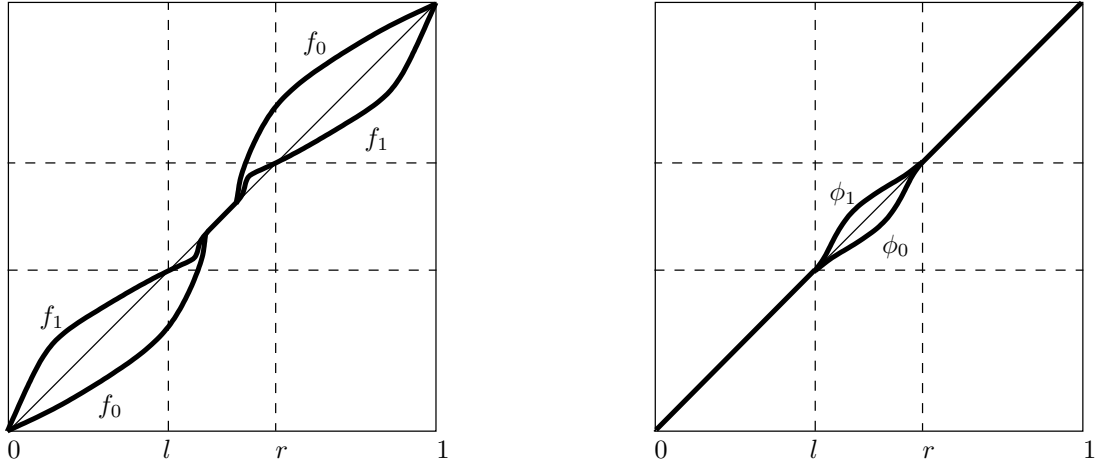
		\caption{\label{f:thick-attractor-attractor} The skew product system corresponding to the iterated function system generated by maps $f_0,f_1$ in left panel, admits two thick attractors. Taking a random walk on its orbits, and adding a composition with additional random choice from maps $\phi_0,\phi_1$ as in the right panel, creates thick metric attractors with intermingled basins.   }
\end{center}
\end{figure}
The iterated function system generated by $f_0$ and $f_1$ provides a skew product system
% $F:\Sigma_2 \times [0,1] \to \Sigma_2 \times [0,1]$
$F:\Sigma_2 \times [0,1] \righttoleftarrow$
given by $F(\omega,x) = (\sigma\omega , f_{\omega_0} (x))$.
Take on $\Sigma_2 \times [0,1]$ a reference measure $\nu_{1/2}\times \lambda$.
The map $F$ has two thick attractors, a thick attractor  $\Lambda_l \subset \Sigma_2 \times [0,l]$ and a thick attractor  $\Lambda_r \subset \Sigma_2 \times [r,1]$.
With $\phi_0,\phi_1$ diffeomorphisms on $[0,1]$ as shown in the right panel of Figure~\ref{f:thick-attractor-attractor},
let
%$H : \Sigma_2 \times  \Omega \times \Sigma_2 \times [0,1] \to \Sigma_2 \times  \Omega \times \Sigma_2 \times [0,1]$
$H : \Sigma_2 \times  \Omega \times \Sigma_2 \times [0,1] \righttoleftarrow$
be given by
 \[
H(\xi,\eta,\omega,x) = \left\{
\begin{array}{ll}
                          \Big( \sigma \xi, \sigma \eta , \sigma \omega, \phi_{\xi_0} \circ f_{\omega_0}(x)\Big), & \eta_0 =  +, \vspace{0.1cm}
                          \\
  \Big( \sigma \xi, \sigma \eta , \sigma^{-1} \omega, \phi_{\xi_0}\circ f_{\omega_{-1}}^{-1}(x)\Big), &  \eta_0 =  -.
\end{array}
\right.
\]
Without the maps $\phi_0,\phi_1$ this is the skew product corresponding to a random walk along orbits of $F$.  The additional compositions with $\phi_0,\phi_1$ provide a random walk between basins of attraction of the two thick metric attractors.

\begin{theorem}\label{t:tma-alternative}
Consider the skew product system $H$  on  $\Sigma_2 \times  \Omega \times \Sigma_2 \times [0,1]$.
Take $\nu_{1/2} \times \nu_p \times \nu_{1/2}\times \lambda$ with $1/2 < p < 1$ as reference measure.
The sets $\Sigma_2  \times  \Omega \times \Lambda_l$ and  $\Sigma_2  \times  \Omega \times \Lambda_r$ are thick metric attractors
for $H $with intermingled basins.
\end{theorem}

\begin{proof}[Sketch of proof]
Write $\Pi: \Sigma_2 \times  \Omega \times \Sigma_2 \times [0,1] \to [0,1]$ for the coordinate projection to the last coordinate in $[0,1]$.
As before in Section~\ref{subs:thick} we find that $\Sigma_2  \times  \Omega \times \Lambda_l$ and  $\Sigma_2  \times  \Omega \times \Lambda_r$ are thick metric attractors.

For any $x \in [0,1]$,
 \begin{align}\label{e:nuto0l>0}
\nu_{1/2} \times \nu_p \times \nu_{1/2} \left\{ (\xi,\eta,\omega) \; ; \; \Pi H^n (\xi,\eta,\omega,x) \in [0,l]  \right\} &> 0,
\end{align}
since one can find a cylinder $C \subset \Sigma_2  \times  \Omega \times \Sigma_2$
so that $\Pi H^n (\xi,\eta,\omega,x) \in [0,l]$ for $(\xi,\eta,\omega) \in C$.
Likewise
\begin{align}\label{e:nutor1>0}
\nu_{1/2} \times \nu_p \times \nu_{1/2} \left\{ (\xi,\eta,\omega) \; ; \; \Pi H^n (\xi,\eta,\omega,x) \in [r,1]  \right\} &> 0.
\end{align}
From this we get that for any open set $U \subset \Sigma_2 \times  \Omega \times \Sigma_2 \times [0,1]$, the basin of attraction of $\Sigma_2  \times  \Omega \times \Lambda_l$ and of  $\Sigma_2  \times  \Omega \times \Lambda_r$ intersect $U$ in a set of positive reference measure.
So $H$  admits two thick metric attractors with intermingled basins.
\end{proof}

\subsubsection{Multiple thick metric attractors with intermingled basins}

We point out possible extensions to skew product systems with multiple thick metric attractors
and also  multiple thick metric attractors with mutually intermingled basins.
\begin{figure}[!ht]
	\begin{center}
        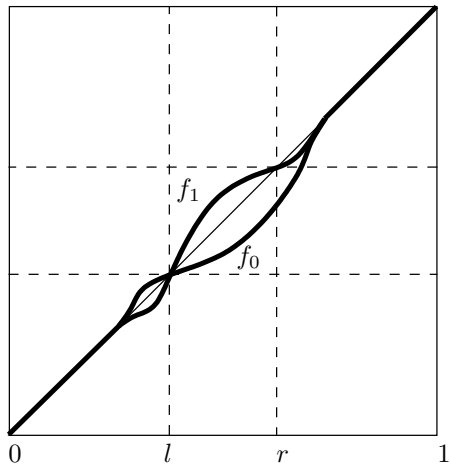
		\caption{\label{f:ingredient} We consider an iterated function system generated by diffeomorphisms $f_0$ and $f_1$ on $[0,1]$ with graphs as depicted. There is an invariant interval $[l,r]$.   }
\end{center}
\end{figure}
Consider an iterated function system on
$[0,1]$ generated by functions $f_0,f_1$ with graphs as depicted in Figure~\ref{f:ingredient}.
There is an invariant interval $I = [l,r]$ (mapped into itself by both maps).
Outside a larger interval $\tilde{I}$ the  maps are the identity maps.

For  a positive integer $K_1$,  take $K_1$ mutually disjoint intervals $\tilde I_1,\ldots,\tilde I_{K_1}$ inside $\mathbb{T}$, with subintervals $I_i \subset \tilde I_i$. Let $f_0, f_1$
be diffeomorphisms on $\mathbb{T}$ that are identity maps outside $\tilde I_1 \medcup \cdots \medcup \tilde I_{K_1}$ and on $\tilde I_i$ have graphs as depicted in Figure~\ref{f:ingredient}. For a second set of intervals
 $J_1 \subset \tilde{J}_1,\ldots,J_{K_2}\subset \tilde{J}_{K_2}$  with the $\tilde{J}_i$'s mutually disjoint, take a similar set of diffeomorphisms
$g_0,g_1$ on $\mathbb{T}$.
Both iterated function systems, generated by $\{ f_0,f_1\}$ and $\{g_0,g_1\}$ have corresponding skew product systems. Take the product of these.
That is, consider the skew product system
%$F : \Sigma_{2}\times\Sigma_{2} \times \mathbb{T}^2 \to   \Sigma_{2}\times\Sigma_{2} \times \mathbb{T}^2$
$F : \Sigma_{2}\times\Sigma_{2} \times \mathbb{T}^2 \righttoleftarrow$
defined by
\begin{align}\label{e:G}
F(\omega , \zeta ,  x , y) &= \left(\sigma \omega , \sigma \zeta , f_{\omega_0} (x) , g_{\zeta_0} (y) \right).
\end{align}
On both of the  $\Sigma_{2}$'s we take $(1/2,1/2)$-Bernoulli measure $\nu_{1/2}$.
The skew product system $(\omega,x)\mapsto (\sigma\omega,f_{\omega_0} (x))$ has $K_1$ thick attractors
$\Lambda_i \subset \Sigma_2 \times I_i$. The skew product system $(\zeta,y)\mapsto (\sigma\zeta,g_{\zeta_0} (y))$ has $K_2$ thick attractors
$\Xi_j \subset \Sigma_2 \times J_j$.
As in \cite{MR2644340} one can show that $F$ admits $K_1 K_2$ thick attractors
$\Lambda_i\times \Xi_j \subset \Sigma_2 \times \Sigma_2 \times I_i \times J_j$, $1 \le i \le K_1, 1 \le j \le K_2$.

To construct examples of systems with thick metric attractors and mutually intermingled basins,
we use the above strategy of introducing random walks along orbits, and compose with additional random walks
outside the squares $I_i \times J_j$, $1 \le i \le K_1, 1 \le j \le K_2$.
Write
\begin{align*}
\mathcal{I}_x = \bigcup_{\{1\le i \le K_1\}} I_{i}\times \mathbb{T},\qquad
\mathcal{I}_y = \bigcup_{\{1\le j \le K_2\}} \mathbb{T} \times I_{j},
\end{align*}
and
\[\mathcal{K}_x = \mathbb{T}^2 \setminus \mathcal{I}_x, \qquad \mathcal{K}_y = \mathbb{T}^2 \setminus \mathcal{I}_y.
\]
Take a diffeomorphism $h_0$, near the identity map,
of the form
$
h_0 (x,y) = (x + \phi^x (x,y), y)
$
where $\phi^x = 0$ on $\mathcal{I}_y$ and $\phi^x  > 0$ on $\mathcal{K}_y$.
Let $h_1 = h_0^{-1}$.
Take another diffeomorphism $h_2$, near the identity map,
of the form
$
h_2 (x,y) = (x, y  + \phi^y (x,y))
$
where $\phi^y = 0$ on $\mathcal{I}_x$ and $\phi^y  > 0$ on $\mathcal{K}_x$.
Let $h_3 = h_2^{-1}$.
For $i = 1,2,3,4$,
let
$H_i:\Sigma_2 \times \Sigma_2 \times \mathbb{T}^2\righttoleftarrow$ be given by
\[
H_i ( \omega,\zeta , x ,y) = (\omega,\zeta, h_i(x,y)).
\]
Define
$H:\Sigma_4 \times \Omega \times \Sigma_2 \times \Sigma_2 \times \mathbb{T}^2\righttoleftarrow$ by
\[
H(\xi , \eta , \omega,\zeta,x,y) = \left(\sigma \xi , \sigma \eta , H_{\xi_0} \circ F^{\eta_0} (\omega,\zeta,x,y)\right).
\]
Endow $\Sigma_4$ with $(1/4,1/4,1/4,1/4)$-Bernoulli measure $\nu_{1/4,1/4,1/4,1/4}$ and $\Omega$ with
$(p,1-p)$-Bernoulli measure $\nu_p$ for $1/2 < p < 1$.

\begin{theorem}
Consider the skew product system $H$  on  $\Sigma_4 \times \Omega \times \Sigma_2 \times \Sigma_2 \times \mathbb{T}^2$.
Take $\nu_{1/4,1/4,1/4,1/4}\times \nu_p \times \nu_{1/2} \times \nu_{1/2} \times \lambda$ as reference measure.
The sets
$\Sigma_4 \times \Omega \times \Lambda_i\times \Xi_j$, $1 \le i \le K_1, 1 \le j \le K_2$ are thick metric attractors
for $H $ with mutually intermingled basins.
\end{theorem}

\begin{proof}[Sketch of proof]
The proof follows the lines of Theorem~\ref{t:tma-alternative}.
Properties \eqref{e:nuto0l>0} and \eqref{e:nutor1>0} get replaced by the following.
For any $(x,y) \in \mathbb{T}^2$, there is a cylinder in $\Sigma_4 \times \Omega \times \Sigma_2 \times \Sigma_2$ so that the torus coordinate of $H^n (\xi , \eta , \omega,\zeta,x,y)$ is in a given square  $I_i \times J_j \subset \mathcal{I}_x \medcap \mathcal{I}_y$.
\end{proof}

%Arguments as before allow to show that the mapping $H$ has $K_1K_2$ thick metric attractors with intermingled basins.

\bibliographystyle{plain}
\bibliography{intermingle}

\end{document}

%% file: ifs-thick.pdf_t
\begin{picture}(0,0)%
\includegraphics{ifs-thick.pdf}%
\end{picture}%
\setlength{\unitlength}{4144sp}%
\begingroup\makeatletter\ifx\SetFigFont\undefined%
\gdef\SetFigFont#1#2#3#4#5{%
  \reset@font\fontsize{#1}{#2pt}%
  \fontfamily{#3}\fontseries{#4}\fontshape{#5}%
  \selectfont}%
\fi\endgroup%
\begin{picture}(2632,2814)(405,-2379)
\put(630,-1522){\makebox(0,0)[lb]{\smash{{\SetFigFont{10}{12.0}{\rmdefault}{\mddefault}{\updefault}{\color[rgb]{0,0,0}$f_1$}%
}}}}
\put(2234,-849){\makebox(0,0)[lb]{\smash{{\SetFigFont{10}{12.0}{\rmdefault}{\mddefault}{\updefault}{\color[rgb]{0,0,0}$f_0$}%
}}}}
\put(3004,-2324){\makebox(0,0)[lb]{\smash{{\SetFigFont{10}{12.0}{\rmdefault}{\mddefault}{\updefault}{\color[rgb]{0,0,0}$1$}%
}}}}
\put(2042,-2324){\makebox(0,0)[lb]{\smash{{\SetFigFont{10}{12.0}{\rmdefault}{\mddefault}{\updefault}{\color[rgb]{0,0,0}$r$}%
}}}}
\put(1368,-2324){\makebox(0,0)[lb]{\smash{{\SetFigFont{10}{12.0}{\rmdefault}{\mddefault}{\updefault}{\color[rgb]{0,0,0}$l$}%
}}}}
\put(438,-2324){\makebox(0,0)[lb]{\smash{{\SetFigFont{10}{12.0}{\rmdefault}{\mddefault}{\updefault}{\color[rgb]{0,0,0}$0$}%
}}}}
\end{picture}%

%% file: ifs-alternative.pdf_t
\begin{picture}(0,0)%
\includegraphics{ifs-alternative.pdf}%
\end{picture}%
\setlength{\unitlength}{4144sp}%
\begin{picture}(6492,2813)(405,-2380)
\put(6874,-2324){\makebox(0,0)[lb]{\smash{\fontsize{10}{12}\usefont{T1}{ptm}{m}{n}{\color[rgb]{0,0,0}$1$}%
}}}
\put(5912,-2324){\makebox(0,0)[lb]{\smash{\fontsize{10}{12}\usefont{T1}{ptm}{m}{n}{\color[rgb]{0,0,0}$r$}%
}}}
\put(5238,-2324){\makebox(0,0)[lb]{\smash{\fontsize{10}{12}\usefont{T1}{ptm}{m}{n}{\color[rgb]{0,0,0}$l$}%
}}}
\put(4308,-2324){\makebox(0,0)[lb]{\smash{\fontsize{10}{12}\usefont{T1}{ptm}{m}{n}{\color[rgb]{0,0,0}$0$}%
}}}
\put(630,-1522){\makebox(0,0)[lb]{\smash{\fontsize{10}{12}\usefont{T1}{ptm}{m}{n}{\color[rgb]{0,0,0}$f_1$}%
}}}
\put(3004,-2324){\makebox(0,0)[lb]{\smash{\fontsize{10}{12}\usefont{T1}{ptm}{m}{n}{\color[rgb]{0,0,0}$1$}%
}}}
\put(2042,-2324){\makebox(0,0)[lb]{\smash{\fontsize{10}{12}\usefont{T1}{ptm}{m}{n}{\color[rgb]{0,0,0}$r$}%
}}}
\put(1368,-2324){\makebox(0,0)[lb]{\smash{\fontsize{10}{12}\usefont{T1}{ptm}{m}{n}{\color[rgb]{0,0,0}$l$}%
}}}
\put(438,-2324){\makebox(0,0)[lb]{\smash{\fontsize{10}{12}\usefont{T1}{ptm}{m}{n}{\color[rgb]{0,0,0}$0$}%
}}}
\put(2574,-456){\makebox(0,0)[lb]{\smash{\fontsize{10}{12}\usefont{T1}{ptm}{m}{n}{\color[rgb]{0,0,0}$f_1$}%
}}}
\put(979,-2059){\makebox(0,0)[lb]{\smash{\fontsize{10}{12}\usefont{T1}{ptm}{m}{n}{\color[rgb]{0,0,0}$f_0$}%
}}}
\put(2206,119){\makebox(0,0)[lb]{\smash{\fontsize{10}{12}\usefont{T1}{ptm}{m}{n}{\color[rgb]{0,0,0}$f_0$}%
}}}
\put(5671,-1096){\makebox(0,0)[lb]{\smash{\fontsize{10}{12}\usefont{T1}{ptm}{m}{n}{\color[rgb]{0,0,0}$\phi_0$}%
}}}
\put(5356,-781){\makebox(0,0)[lb]{\smash{\fontsize{10}{12}\usefont{T1}{ptm}{m}{n}{\color[rgb]{0,0,0}$\phi_1$}%
}}}
\end{picture}%

%% file: ifs-ingredient.pdf_t
\begin{picture}(0,0)%
\includegraphics{ifs-ingredient.pdf}%
\end{picture}%
\setlength{\unitlength}{4144sp}%
\begin{picture}(2625,2814)(406,-2380)
\put(3004,-2324){\makebox(0,0)[lb]{\smash{\fontsize{10}{12}\usefont{T1}{ptm}{m}{n}{\color[rgb]{0,0,0}$1$}%
}}}
\put(2042,-2324){\makebox(0,0)[lb]{\smash{\fontsize{10}{12}\usefont{T1}{ptm}{m}{n}{\color[rgb]{0,0,0}$r$}%
}}}
\put(1368,-2324){\makebox(0,0)[lb]{\smash{\fontsize{10}{12}\usefont{T1}{ptm}{m}{n}{\color[rgb]{0,0,0}$l$}%
}}}
\put(438,-2324){\makebox(0,0)[lb]{\smash{\fontsize{10}{12}\usefont{T1}{ptm}{m}{n}{\color[rgb]{0,0,0}$0$}%
}}}
\put(1801,-1141){\makebox(0,0)[lb]{\smash{\fontsize{10}{12}\usefont{T1}{ptm}{m}{n}{\color[rgb]{0,0,0}$f_0$}%
}}}
\put(1441,-736){\makebox(0,0)[lb]{\smash{\fontsize{10}{12}\usefont{T1}{ptm}{m}{n}{\color[rgb]{0,0,0}$f_1$}%
}}}
\end{picture}%